\documentclass{article}

\usepackage{amsfonts}
\usepackage{amssymb}
\usepackage{amsthm}
\usepackage{amsmath}
\usepackage{epsfig}
\usepackage{psfrag}
\usepackage{url}
\usepackage{mathdots}
\usepackage{xcolor}

\def\@begintheorem#1#2{\par\bgroup{\sc #1\ #2. }\it\ignorespaces}
\def\@opargbegintheorem#1#2#3{\par\bgroup{\sc #1\ #2\ (#3). } \it\ignorespaces}
\def\@endtheorem{\egroup}
\newtheorem{theorem}{Theorem}[section]
\newtheorem{corollary}[theorem]{Corollary}
\newtheorem{lemma}[theorem]{Lemma}
\newtheorem{proposition}[theorem]{Proposition}

\newtheorem{example}[theorem]{Example}
\newtheorem{remark}[theorem]{Remark}

\newtheorem{conjecture}[theorem]{Conjecture}

\newtheorem{*theorem}[theorem]{*Theorem}
\newtheorem{*corollary}[theorem]{*Corollary}
\newtheorem{*lemma}[theorem]{*Lemma}
\newtheorem{*proposition}[theorem]{*Proposition}

\def\N{\mathbb{N}}

\def\R{\mathbb{R}}
\def\real{\mathbb{R}}
\def\reals{\mathbb{R}}

\def\aff{\operatorname{aff}}
\def\conv{\operatorname{conv}}
\def\diam{\operatorname{diam}}
\def\ast{\operatorname{ast}}
\def\lk{\operatorname{lk}}

\def\wed{\operatorname{W}}
\def\ops{\operatorname{S}}

\begin{document}

\title{An update on the Hirsch conjecture
}

\author{
Edward D. Kim\thanks{Supported in part by the Centre de Recerca Matem\`atica, NSF grant DMS-0608785 and NSF VIGRE grants DMS-0135345 and DMS-0636297.}
\and
Francisco Santos\thanks{Supported in part by the Spanish Ministry of Science through grant MTM2008-04699-C03-02}
}

\date{}

\maketitle

\begin{abstract}
The Hirsch conjecture was posed in 1957 in a question from Warren M. Hirsch to George Dantzig. It states that the graph of a $d$-dimensional polytope with $n$ facets cannot have diameter greater than $n-d$. 

Despite being one of the most fundamental, basic and old problems in polytope theory, what we know is quite scarce. Most notably, no polynomial upper bound is known for the diameters that are conjectured to be linear. In contrast, very few polytopes are known where the bound $n-d$ is attained. This paper collects known results and remarks both on the positive and on the negative side of the conjecture. Some proofs are included, but only those that we hope are accessible to a general mathematical audience without introducing too many technicalities.
\end{abstract}

\section{Introduction}

Convex polytopes generalize convex polygons (of dimension two).
More precisely, a \emph{convex polyhedron} is any intersection of finitely many affine closed half-spaces in $\real^d$. A \emph{polytope} is a bounded polyhedron.  The long-standing Hirsch conjecture is the following very basic statement about the structure of polytopes. Besides its implications in linear programming, which motivated the conjecture in the first place, it is one of the most fundamental open questions in polytope theory.

\begin{conjecture}[Hirsch conjecture]\label{conj:hirsch}
Let $n > d \geq 2$. Let $P$ be a $d$-dimensional polytope with $n$ facets.  Then $\diam(G(P)) \leq n - d$.
\end{conjecture}

\emph{Facets} are the faces of dimension $d-1$ of $P$, so that
the number $n$ of them is the minimum number of half-spaces needed to produce $P$ as their intersection (assuming $P$ is full-dimensional). The number $\diam(G(P))\in \N$ is the \emph{diameter of the graph of $P$}. Put differently, the conjecture states that we can go from any vertex of $P$ to any other vertex traversing at most $n-d$ edges. 

Consider the following examples; all of them satisfy the inequality strictly, except for the cube where it is tight:
\begin{center}
\begin{tabular}{c|ccc}
$P$ & $n$ & $d$ & $\diam(G(P))$ \cr
\hline
polygon & $n$ & $2$ & $\lfloor n/2 \rfloor$ \\
cube & $6$ & $3$ & $3$ \\
icosahedron & $20$ & $3$ & $3$ \\
soccer-ball & $32$ & $3$ & $9$ \\
\end{tabular}
\end{center}

Polytopes and polyhedra are the central objects in the area of geometric combinatorics, but they also appear in diverse mathematical fields:
From the applications point of view, a polyhedron is the \emph{feasibility region} of a \emph{linear program}~\cite{Dantzig-book}. This is the context in which the Hirsch conjecture was originally posed (see below).
 In \emph{toric geometry}, to every (rational) polytope one associates a certain projective variety (see, e.g.,~\cite{Oda-book}). The underlying interaction between combinatorics and algebraic geometry has proved extremely fruitful for both areas, leading for example to a complete characterization of the possible numbers of faces (vertices, edges, facets, \dots) that a \emph{simplicial polytope} can have. The same question for arbitrary polytopes is open in dimension four and higher~\cite{Ziegler-facenumbers}. Polytopes with special symmetries, such as regular ones and variations of them arise naturally from Coxeter groups and other algebraic structures~\cite{Bjorner-Brenti, FominZelevinski-associahedra}. Last but not least, counting integer points in polytopes with integer vertex coordinates has applications ranging from number theory and
representation theory to cryptography, integer programming, and statistics~\cite{Beck-Robins,deloera-latticepoints}.

In this paper we review the current status of the Hirsch conjecture and related questions. Some proofs are included, and many more appear in an appendix which is available electronically~\cite{Kim-Santos-companion}. Results whose proof can be found in~\cite{Kim-Santos-companion} are marked with an asterisk. 
An earlier survey of this topic, addressed to a more specialized audience, was written by Klee and Kleinshmidt in 1987~\cite{Klee-Kleinschmidt}.

\subsection{A bit of polytope theory}

We now review several concepts that will appear throughout this paper. For further discussion, we refer the interested reader to~\cite{Ziegler:LecturesPolytopes}.

A \emph{polyhedron} is the intersection of a finite number of closed half-spaces and a polytope is a bounded polyhedron. A \emph{polytope} is, equivalently, the convex hull of a finite collection of points. Although the geometric objects are the same, from a computational point of view it makes a difference whether a certain polytope is represented as a convex hull or via linear inequalities: the size of one description cannot be bounded polynomially in the size of the other, if the dimension $d$ is not fixed. The \emph{dimension} of a polytope is the dimension of its affine hull $\aff(P)$. A $d$-dimensional polytope is called a \emph{$d$-polytope}.

If $H$ is a closed half-space containing $P$, then the intersection of $P$ with the boundary of $H$ is called a \emph{face} of $P$. A non-empty face is the intersection of $P$ with a supporting hyperplane. Faces are themselves polyhedra of lower dimension. A face of dimension $i$ is called an \emph{$i$-face}. The $0$-faces are the \emph{vertices} of $P$, the $1$-faces are \emph{edges}, the $(d-2)$-faces are \emph{ridges}, and the $(d-1)$-faces are called \emph{facets}. In its irredundant description, a polytope is the convex hull of its vertices, and the intersection of its facet-defining half-spaces.

For a polytope $P$, we denote by $G(P)$ its \emph{graph} or \emph{$1$-skeleton}, consisting of the vertices and edges of $P$: the vertices of the graph $G(P)$ are indexed by the vertices of the polytope $P$, and two vertices in the graph $G(P)$ are connected by an edge exactly when their corresponding vertices in $P$ are contained in a $1$-face. The \emph{distance} between two vertices in a graph is the minimum number of edges needed to go from one to the other, and the \emph{diameter} of a graph is the maximum distance between its vertices. Let $H(n,d)$ denote the maximum diameter of graphs of $d$-polytopes with $n$ facets. (For an unbounded polyhedron, the graph contains only the \emph{bounded} edges. The unbounded $1$-faces are called \emph{rays}.)

\begin{example}\label{example:basic}
Examples of polytopes one can build in every dimension are the following:
\begin{enumerate}

\item {\bf The $d$-simplex.}
The convex hull of $d+1$ points in $\R^d$ that do not lie on a common hyperplane is a $d$-dimensional simplex. It has $d+1$ vertices and $d+1$ facets. Its graph is complete, so its diameter is 1.

\item {\bf The $d$-cube.}
The vertices of the $d$-cube, the product of $d$ segments,  are the $2^d$ points with $\pm 1$ coordinates. Its facets are given by the $2d$ inequalities $-1\le x_i\le 1$.  Its graph has diameter $d$: the steps needed to go from a vertex to another equals the number of coordinates in which the two vertices differ. 

\item {\bf Cross polytope.}
This  is the convex hull of the $d$ standard basis vectors and their negatives, which
generalizes the $3$-dimensional octahedron. It has $2^d$ facets,  one in each orthant of $\reals^d$. Its graph is almost complete: the only edges missing from it are those between opposite vertices.
 \end{enumerate}
See Figure \ref{fig:basic-examples}

\begin{figure}[htb]
\begin{center}
\includegraphics[width=3in]{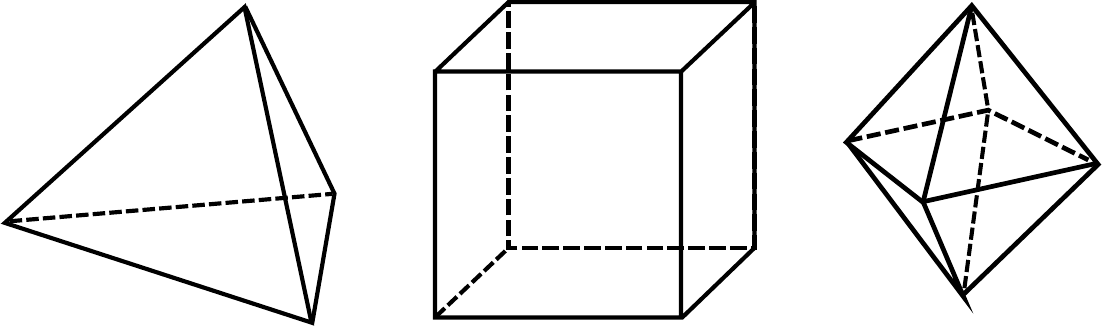}
\caption{Basic examples of polytopes}
\label{fig:basic-examples}
\end{center}
\end{figure}

\end{example}
Their numbers $m$ of vertices, $n$ of facets, dimension $d$ and diameter are:
\begin{center}
\begin{tabular}{c|cccc}
$P$ & $m$&$n$ & $d$ & $\diam(G(P))$ \cr
\hline
$d$-simplex &$d+1$ & $d+1$ & $d$ & $1$ \\
$d$-cube &$2^d$ & $2d$ & $d$ & $d$ \\
$d$-crosspolytope &$2d$ & $2^d$ & $d$ & $2$ \\
\end{tabular}
\end{center}

Of special importance are the simple and simplicial polytopes. A $d$-polytope is called \emph{simple} if every vertex is the intersection of exactly $d$ facets. Equivalently, a $d$-polytope is simple if every vertex in the graph $G(P)$ has degree exactly $d$. We note that the $d$-simplices and $d$-cubes are simple, but cross-polytopes are not simple starting in dimension three. Any polytope or polyhedron $P$, given by its facet-description, can be perturbed to a simple one $P'$ by a generic and small change in the coefficients of its defining inequalities. This will make non-simple vertices ``explode'' and become clusters of new vertices, all of which will be simple. This process can not decrease the diameter of the graph, since we can recover the graph of $P$ from that of $P'$ by collapsing certain edges. Hence, to study the Hirsch conjecture, one only needs to consider the simple polytopes:
\begin{lemma}
\label{lemma:simple}
The diameter of any polytope $P$ is bounded above by the diameter of some simple polytope $P'$ with the same dimension and number of facets.
\end{lemma}

Graphs of simple polytopes are better behaved than graphs of arbitrary polytopes. Their main property in the context of the Hirsch conjecture is that if $u$ and $v$ are vertices joined by an edge in a simple polytope then there is a single facet containing $u$ and not $v$, and a single facet containing $v$ and not $u$. That is, at each step along the graph of $P$ we enter a single facet and leave another one. 

Every polytope $P$ (containing the origin in its interior, which can always be assumed by a suitable translation) has a polar polytope $P^*$ whose vertices (respectively facets) correspond to the facets (respectively vertices) of $P$. More generally, every $(d-i)$-face of $P^*$ corresponds to a face of $P$ of dimension $i-1$, and the incidence relations are reversed.

The polars of simple polytopes are called \emph{simplicial}, and their defining property is that every facet is a $(d-1)$-simplex. As an example, the $d$-dimensional cross polytope is the polar of the $d$-cube. Since cubes are simple polytopes, cross polytopes are simplicial. The polar of a simplex is a simplex, and simplices are the only polytopes of dimension greater than two which are at the same time simple and simplicial. Since all faces of a simplex are themselves simplices, all faces of a simplicial polytope are simplices. From this viewpoint, one can forget the geometry of $P^*$ and look only at the combinatorics of the simplicial complex formed by its faces, the boundary of $P$. Topologically, this simplicial complex is a sphere of dimension $d-1$.

For simplicial polytopes we can state the Hirsch conjecture as asking how many ridges do we need to cross in order to walk between two arbitrary facets, if we are only allowed to move from one facet to another via a ridge. This suggests defining the \emph{dual graph} $G^\Delta(P)$ of a polytope: The undirected graph having as nodes the facets of $P$ and in which two nodes are connected by an edge if and only if their corresponding facets intersect in a ridge of $P$. In summary, $G^\Delta(P)=G(P^*)$.

\subsection{Relation to Linear Programming}

The  original motivation for the Hirsch conjecture comes from its relation to the simplex algorithm for linear programming. In linear programming, one is given a system of linear equalities and inequalities, and the goal is to maximize (or minimize) a certain linear functional. Every such problem can be put in the following standard form, where $A$ is an $m \times n$ real matrix $A$, and ${\bf b}\in \reals^m$ and ${c}\in \reals^n$ are  two real vectors:
\[
\text{ Maximize } c\cdot {\bf x}, \text{ subject to } A{\bf x} = {\bf b} \text{ and } {\bf x} \ge 0.
\]
Suppose the matrix $A$ has full row rank $m \leq n$. Then, the equality $A{\bf x} = {\bf b}$ defines a $d$-dimensional affine subspace ($d=n-m$), whose intersection with the linear inequalities $ {\bf x} \ge 0$ gives the \emph{feasibility polyhedron} $P$:
\[
P:= \{ {\bf x} \in \R^n : A{\bf x} = {\bf b} \text{ and } {\bf x} \ge 0\}.
\]
One typically desires not only the maximum value of $c\cdot {\bf x}$ but also (the actual coordinates of) a vector ${\bf x} \in P$ where the maximum is attained. It is easy to prove that such an $x$, if it exists, can be found among the vertices of $P$. If $P$ is unbounded and $c\cdot {\bf x}$ does not have an upper bound on it
one considers the problem ``solved'' by describing a ray of $P$ where the value $c \dot {\bf x}$ goes to infinity. 

In 1979, Khachiyan~\cite{Khachiyan} proved that linear programming problems can be solved in polynomial time via the so-called \emph{ellipsoid method}. In 1984, Karmarkar~\cite{Karmarkar} devised a different approach,  the \emph{interior point method}. 
Although the latter is more applicable (easier to implement, better complexity) than the former, still to this day the most commonly used method for linear programming is the \emph{simplex method} devised by G. Dantzig in 1947. 
For a complete account of the complexity of linear programming, see the survey~\cite{Megiddo:SurveyLPCompelxity} by Megiddo.

In geometric terms, the simplex method first finds an arbitrary vertex in the feasibility polyhedron $P$. Then, it moves from vertex to adjacent vertex in such a way that the value $c \cdot {\bf x}$ of the linear functional increases at every step. 
These steps are called \emph{pivots} and the rule used to choose one specific adjacent vertex is called the \emph{pivot rule}.
When no {pivot step} can increase the functional, convexity implies that we have arrived to the global maximum. 

Clearly, a lower bound for the performance of the simplex method under \emph{any} pivot rule is the diameter of the polyhedron $P$. The converse is not true, since knowing that $P$ has a small graph diameter does not in principle tell us how to go from one vertex to another in a small number of steps. In particular, many of the results on diameters of polyhedra do not help for the simplex method.

In fact, the complexity of the simplex method depends on the local rule (known as a \emph{pivot rule}) chosen to move from vertex to vertex. The \emph{a priori} best pivot rule, the one originally proposed by Dantzig,  is ``move along the edge with maximum gradient'', but Klee and Minty~\cite{Klee-Minty} showed in 1972 that this can lead to paths of exponential length, even in polytopes combinatorially equivalent to cubes. The same worst-case exponential behavior has been proved for essentially every deterministic rule devised so far, although there are subexponential, but yet not polynomial, randomized pivot algorithms (see Theorem~\ref{thm:randomizedpivot}). However, the simplex algorithm is highly efficient \emph{in practice} on most linear optimization problems. 

There is another reason why investigating the complexity of the simplex method is important, even if we already know polynomial time algorithms. The algorithms of Khachiyan and Karmarkar are polynomial in the \emph{bit length} of the input; but it is of practical importance to know whether a polynomial algorithm for linear programming in the \emph{real number machine} model of Blum, Cucker, Shub, and Smale~\cite{BCSS} exists. That is, is there an algorithm that uses a number of arithmetic operations that is polynomial on the number of
coefficients of the linear program, rather than on their total bit-length; or, better yet, a \emph{strongly polynomial algorithm}, i.e. one that is polynomial both in the arithmetic sense and the bit sense? These two related problems were included by Smale in his list of ``mathematical problems for the next century''~\cite{Smale}. A polynomial pivot rule for the simplex method would solve them in the affirmative.

In this context, the following {\em polynomial version} of the conjecture is relevant, if the linear one turns out to be false. See, for example,~\cite{kalaiblog}:

\begin{conjecture}[Polynomial Hirsch conjecture]\label{conj:hirsch-poly}
Is there a polynomial function $f(n,d)$ such that for any polytope (or polyhedron)
$P$ of dimension $d$ with $n$ facets, $\diam(G(P)) \leq f(n,d)$?.
\end{conjecture}

\subsection{Overview of this paper}

Our initial purpose with this paper was two-fold: on the one hand, we thought it is about time to have in a single source an overview of the state of the art  concerning the Hirsch conjecture and related issues, putting up to date the 20 year old survey by Klee and Kleinschmidt~\cite{Klee-Kleinschmidt}.
On the other hand, since there seems to be agreement on the fact that the Hirsch conjecture is probably false (with opinions about the polynomial version of it being divided) we wanted to give a fresh look at the past attempts to \emph{disprove} the conjecture. 

These two goals turned out to be in conflict, or at least too ambitious, so the first version of the paper  was too long and too technical for the intended readership. After wise comments from our editor J\"org Rambau and an anonymous referee, we decided to take most of the proofs out of the main paper, and compiled them in the companion paper~\cite{Kim-Santos-companion}. Results whose proof can be found in~\cite{Kim-Santos-companion} are marked with an asterisk.

Our two-fold intentions are still reflected in the two quite distinct parts that the paper has, Sections~\ref{sec:bounds-algorithms} and~\ref{sec:constructions}. The first one is devoted to positive results, comprising general upper bounds for polytope diameters, special cases where the Hirsch conjecture is known, etc, and the second one contains mainly constructions and results aimed at disproving the Hirsch conjecture. 

The two sections differ  in several other respects: Section~\ref{sec:bounds-algorithms} is written in an informative style. No proofs are included (although some appear in~\cite{Kim-Santos-companion})
since this section covers quite different topics and the techniques and ideas used are too technical and, more importantly, too diverse. In Section~\ref{sec:constructions}, in the contrary, we provide proofs for essentially all the results (some here and some in~\cite{Kim-Santos-companion}); on the one hand the tools needed are more homogeneous and elementary; on the other hand, in this section we feel  that having a new look at old results is useful. We have tried to identify the basic ingredients in each construction, obtaining in some cases much simpler (to our taste) proofs and expositions than the original ones. In particular, the main novelty in this survey, if any, is probably in our descriptions of the non-Hirsch polyhedron and Hirsch-sharp polytope found by Klee and Walkup in 1967 (Section~\ref{sec:Q4}) and of the non-Hirsch sphere found by Mani and Walkup in 1980 (Section~\ref{sec:topological}). 

Another difference between the two sections is that Section~\ref{sec:bounds-algorithms} contains several very recent developments, while all of Section~\ref{sec:constructions}, with the single exception of Theorem~\ref{thm:hirsch-sharp}, refers to results that are at least 25 years old.

Let us now give a brief roadmap for the paper.
\medskip

Section~\ref{sec:small-dimension} lists the pairs of parameters $(n,d)$ such that the Hirsch Conjecture is known to hold for all $d$-polytopes with $n$ facets. That is, denoting $H(n,d)$ the maximum diameter of $d$-polytopes with $n$ facets, we list all pairs for which the Hirsch inequality $H(n,d)\le n-d$ is known to hold. This comprises the cases $d\le 3$ (Klee~\cite{Klee:PathsII}), $n-d\le 6$ (Bremner and Schewe~\cite{Bremner:DiameterFewFacets}), and $(n,d)\in\{(11,4), (12,4)\}$
(Bremner et al.~\cite{Bremner:MoreBounds}).

The next section lists general upper bounds on $H(n,d)$: a linear one in fixed dimension (Barnette and Larman) ~\cite{Barnette, Larman} and a quasi-polynomial one of $n^{\log_2(d) + 1}$ (Kalai-Kleitman~\cite{Kalai:Quasi-polynomial}). These bounds hold not only for diameters of polytopes but also for much more abstract and general objects. A very recent development by Eisenbrand, H\"ahnle, Razborov and Rothvo\ss~\cite{Eisenbrand:limits-of-abstraction} is the identification of one such class for which the proofs of these two bounds work but which admit objects with quadratic diameter. This may be considered evidence against the Hirsch conjecture.

In Section~\ref{sec:subexponential} we concentrate on algorithmic aspects. For example, we state two algorithmic analogues of the two bounds mentioned above: the proof by Meggido~\cite{Megiddo:LPinLT} that linear programming can be done in linear time if the dimension is fixed, and randomized pivot rules for the simplex method in arbitrary dimension that finish in $O(\exp(K\sqrt{d \log n}))$ (Kalai~\cite{Kalai:A-subexponential-randomized}, Matou\v{s}ek, Sharir and Welzl~\cite{Matousek:A-subexponential-bound}).

We then turn our attention to special polytopes for which good bounds are known. Polytopes with $0$-$1$ coordinates and linear programming duals of transportation polytopes are known to satisfy the Hirsch conjecture (Naddef~\cite{Naddef:Hirsch01}, Balinski~\cite{Balinski:DualTransportation}). Diameters of network-flow polytopes, which include transportation polytopes, have quadratic bounds~\cite{Cunningham:Theoretical-properties,Goldfarb:Polynomial-simplex,Orlin:PolytimeNetworkSimplex}.

Section~\ref{sec:bounds-algorithms} finishes with an account of recent work of Deza, Terlaky and Zinchenko~\cite{Deza:Central-path,Deza:The-continuous-d-step,Deza:Curvature} on a continuous analogue of the Hirsch conjecture that arises in the context of interior point methods for linear programming.
\medskip

Almost all of the results in Section~\ref{sec:constructions} revolve around two basic ingredients. The first one is the~\emph{wedge} operation, which we describe in Section~\ref{sec:wedging}. Wedging is a very simple operation that increases both the dimension and number of facets of a polytope by one maintaining (or increasing) its diameter. Using it it is easy to prove the following fundamental result: 

\begin{theorem}[Klee and Walkup~\cite{Klee:d-step}]
\label{thm:dstep-intro}
$H(d+k,d) \le H(2k,k)$, with equality if (but not only if) $k< d$.
\end{theorem}

In particular,  to prove (or disprove) the Hirsch conjecture one can concentrate in the case where the number of facets equals twice the dimension. This case is sometimes referred to as the \emph{$d$-step Conjecture}, since the Hirsch conjecture is saying that we can go from any vertex to any other vertex in $d$-steps. Via wedging, the Hirsch conjecture is also equivalent to the following \emph{non-revisiting Conjecture}:
If $u$ and $v$ are two arbitrary vertices of a simple polytope $P$, then there is a path from $u$  to $v$ which at every step enters a facet of $P$ that was not visited before. We prove the equivalence of the three conjectures (Hirsch, $d$-step and non-revisiting) in Section~\ref{sec:equivalences}.

The second ingredient is the  construction by Klee and Walkup~\cite{Klee:d-step} of a 4-dimensional polytope with 9 facets that meets the Hirsch bound with equality (that is, whose diameter equals $9-4=5$). Polytopes with this property are called \emph{Hirsch-sharp}. They are easy to construct with a number of facets not exceeding twice their dimension (e.g., cubes). Klee and Walkup's Hirsch-sharp polytope is the smallest ``non-trivial'' Hirsch-sharp polytope, with more facets than twice its dimension. In fact, it is also the starting block to the construction of every other Hirsch-sharp polytope with $n>2d$ known to date. 
In Section~\ref{sec:Q4} we give our own description and coordinatization (much smaller than the original one) of the Klee-Walkup polytope. 

In Section~\ref{sec:hirsch-sharp} we recount the state of the art on the existence of Hirsch-sharp polytopes, following work of Fritzsche, Holt and Klee~\cite{Fritzsche99morepolytopes,Holt:Hsharpd7,Holt:Many-polytopes}. Their results, combined with what is  known for small dimension or number of facets, are summarized in Table~\ref{table:hirsch-tight-table}, which gives a  ``plot'' of the function $H(n,d) - (n-d)$. The horizontal coordinate is $n-2d$, so that the column labelled ``0'' corresponds to the polytopes relevant to the $d$-step conjecture. The cases where we know $H(n,d)$ exactly are marked ``$=$'' or ``$<$'' depending on whether Hirsch-sharp polytopes exist or not. The cases where Hirsch-sharp polytopes are known to exist but for which the Hirsch conjecture is not proved are marked ``$\ge$''. Cases where we neither know the Hirsch conjecture nor the existence of Hirsch-sharp polytopes are marked ``?'' and appear only in dimensions 4, 5 and 6. The diagonal dots in the left column reflect the equality case of Theorem~\ref{thm:dstep-intro}.
\begin{table}[htb]
\begin{center}
\begin{tabular}{|c|cccccccccc|}
\hline
$n -2d$& $\cdots$&0&1&2&3&4&5&6&7& $\cdots$\\
\hline
$\underline{\quad\ d\ \quad}$ &&&&&&&&&&\\
2  &&$=$ & $<$ & $<$ & $<$ & $<$ & $<$ & $<$ & $<$  & $\cdots $ \\
3  &$\iddots$&$=$& $<$ & $<$ & $<$  & $<$ & $<$ & $<$ & $<$ & $\cdots $ \\
4  &$\iddots$&$=$& $=$ & $<$ & $<$ & $<$ & ? & ? & ? & $\cdots $   \\
5  &$\iddots$&$=$ &$=$&?&?&?&?&?&?&$\cdots $\\
6  &$\iddots$&$=$ &?&?&?&?&?&?&?&$\cdots $\\
7  &$\iddots$& $\ge$ & $\ge$ & $\ge$ & $\ge$ & $\ge$ & $\ge$  & $\ge$ & $\ge$ & $\cdots $   \\
8  &$\iddots $& $\ge$ & $\ge$ & $\ge$ & $\ge$ & $\ge$ & $\ge$  & $\ge$ & $\ge$ & $\cdots $   \\
$\vdots$   &$\iddots$&$\vdots$&$\vdots$&$\vdots$&$\vdots$&$\vdots$&\ \ $\vdots$\ \ &$\vdots$\ \ &$\vdots$& $\ddots$\\
\hline
\end{tabular}
\caption{$H(n,d)$ versus $n-d$, the state of the art}
\end{center}
\end{table}\label{table:hirsch-tight-table}

The Klee-Walkup polytope is also instrumental in the construction by Klee, Walkup an Todd~\cite{Klee:d-step, Todd:MonotonicBoundedHirsch} of counter-examples to two generalizations of the Hirsch conjecture that are quite natural in the context of linear programming: the case of perhaps-\emph{unbounded} polyhedra (which was the original conjecture by Hirsch) and a \emph{monotone} version in which we look at the maximum number of monotone steps with respect to a given linear function that are needed to go from any vertex of a polytope $P$ to an optimal vertex.
We show these constructions in Section~\ref{sec:unbounded-monotone}, and show in Section~\ref{sec:topological} a counter-example, by Mani and Walkup~\cite{Mani:A-3-sphere-counterexample} to a third, topological, version of the conjecture.

\section{Bounds and algorithms}
\label{sec:bounds-algorithms}

In this section, we present 
special cases for which the Hirsch conjecture holds, upper bounds for diameters of polytopes and subexponential complexity results for the simplex method. We also summarize recent work on analogues of the conjecture for hyperplane arrangements and for paths of interior point methods.

\subsection{Small dimension or few facets}\label{sec:small-dimension}

The following statements exhaust all pairs $(n,d)$ for which the maximum diameter $H(n,d)$ of $d$-polytopes with $n$ facets is known. We omit the cases $n < 2d$, because  $H(d+k,d)=H(2k, k)$ for all $k<d$ (see Theorem~\ref{thm:dstep-intro}), and the trivial case $d\le 2$. Remember that an asterisk in front of a statement denotes the proof can be found in~\cite{Kim-Santos-companion}.

\begin{*theorem} [Klee~\cite{Klee:PathsII}]
\label{thm:hirschford3}
$H(n,3) = \lfloor \frac{2n}{3} \rfloor - 1$.
\end{*theorem}

\begin{theorem} 
\label{thm:hirschforsmalln}
\begin{itemize}
\item $H(8,4)=4$ (Klee~\cite{Klee:PathsII}).
\item $H(9,4)=H(10,5)=5$ (Klee-Walkup~\cite{Klee:d-step}).
\item  $H(10,4)=5$, $H(11,5)=6$ (Goodey~\cite{Goodey}).
\item $H(11,4)=H(12,6)=6$ (Bremner-Schewe~\cite{Bremner:DiameterFewFacets}).
\item $H(12,4)=7$ (Bremner et al.~\cite{Bremner:MoreBounds}).
\end{itemize}
\end{theorem}

Since $\max_d H(d+k,d)=H(2k,k)$ (see Theorem~\ref{thm:dstep-intro} again), the results for $H(8,4)$, $H(10,5)$ and $H(12,6)$ imply:

\begin{corollary} 
The Hirsch conjecture holds for polytopes with at most six facets more than their dimension.
\end{corollary}

It is easy to generalize one direction of Theorem~\ref{thm:hirschford3}, giving the following lower bound for $H(n,d)$. Observe that the formula gives the exact value of $H(n,d)$ for $d\in\{1,2,3\}$.
\begin{*proposition}
\[
H(n,d) \ge \left \lfloor\frac{d-1}{d} n \right\rfloor - (d-2).
\]
\end{*proposition}

\subsection{General upper bounds on diameters}\label{sec:upper-bounds}

Diameters of polytopes admit a \emph{linear} upper bound when the dimension $d$ is fixed. This was first noticed by  Barnette~\cite{Barnette}  and then improved by Larman~\cite{Larman}:

\begin{*theorem}[Larman~\cite{Larman}]
\label{thm:linear-in-fixed-d}
For every $n>d\ge 3$, $H(n,d)\le n 2^{d-3}$.
\end{*theorem}

But when the number of facets is not much bigger than $d$, a much better upper bound was given by Kalai and Kleitman~\cite{Kalai:Quasi-polynomial}, with a surprisingly simple and elegant proof (the paper is just two pages!).
\begin{*theorem}[Kalai-Kleitman~\cite{Kalai:Quasi-polynomial}]
\label{thm:quasipolynomial}
For every $n>d$, $H(n,d)\le n^{\log_2(d)+1}$.
\end{*theorem}

The proofs of Theorems~\ref{thm:quasipolynomial} and \ref{thm:linear-in-fixed-d} use very limited properties of graphs of polytopes. For example, Klee and Kleinschmidt (see \S 7.7 in \cite{Klee-Kleinschmidt}) show that Theorem~\ref{thm:linear-in-fixed-d} holds for the ridge-graphs of all pure simplicial complexes, and even more general objects. In the same vein, Eisenbrand, H\"ahnle, Razborov and Rothvo\ss~\cite{Eisenbrand:limits-of-abstraction} have recently shown the following generalization of Theorems~\ref{thm:quasipolynomial} and~\ref{thm:linear-in-fixed-d}:

\begin{theorem}[Eisenbrand et al.~\cite{Eisenbrand:limits-of-abstraction}]
\label{thm:eisenbrand}
Let $G$ be a graph whose vertices are certain  subsets of size $d$ of an $n$-element set. Assume that between every pair of vertices $u$ and $v$ in $G$ there is a path  using only vertices that contain $u \cap v$.  

Then,
$\diam(G)\le \min\{n^{1+\log d},  n 2^{d-1} \}$.
\end{theorem}

The novelty in~\cite{Eisenbrand:limits-of-abstraction} is that the authors show that there are graphs with the hypotheses of Theorem~\ref{thm:eisenbrand} and with $\diam(G)\ge c n^2/\log n$, for arbitrarily large $n$ and a certain constant $c$. It is not clear whether this is support against the Hirsch conjecture or it simply indicates that the arguments in the proofs of Theorems~\ref{thm:quasipolynomial} and \ref{thm:linear-in-fixed-d} do not take advantage of properties that graphs of polytopes have and which prevent their diameters from growing. For example, observe that \emph{any} connected graph is valid for the case $d=1$ of Theorem~\ref{thm:eisenbrand}. 


\subsection{Subexponential simplex algorithms}
\label{sec:subexponential}
Since the Hirsch conjecture is strongly motivated by the simplex algorithm of linear programming, it is natural to ask about the number of iterations needed under particular pivot rules. Most of the proofs for the upper bounds in the previous sections do not give a clue on how to find a short path towards the vertex maximizing a given functional, or even an explicit path between any pair of given vertices. 

Kalai~\cite{Kalai:A-subexponential-randomized} and, independently, Matou\v{s}ek, Sharir and Welzl~\cite{Matousek:A-subexponential-bound}
proved the existence of randomized pivot rules for the simplex method with subexponential running time for arbitrary linear programs.
\begin{theorem}[Kalai~\cite{Kalai:A-subexponential-randomized}, Matou\v{s}ek, Sharir and Welzl~\cite{Matousek:A-subexponential-bound}]
\label{thm:randomizedpivot}
There exist randomized simplex algorithms where the expected number of arithmetic operations needed in the worst case is at most $\exp(K\sqrt{d \log n})$, where $K$ is a fixed constant.
\end{theorem}

If we consider Theorem~\ref{thm:randomizedpivot} as an algorithmic analogue of Theorem~\ref{thm:quasipolynomial}, then the following result of Megiddo is the analogue of Theorem~\ref{thm:linear-in-fixed-d}. It says that linear programming can be performed in linear time in fixed dimension:

\begin{theorem}[Meggido~\cite{Megiddo:LPinLT}]
\label{thm:LPinLT}
There are pivot rules  for the simplex algorithm that run in $O(2^{2^d}n)$ time.
\end{theorem}

It is also known that random polytopes have polynomial diameter, which explains why the simplex method seems to work well in practice. The first results in this direction were
proved by  Borgwardt~\cite{Borgwardt} and, independently, Smale~\cite{Smale-AverageLP}, who analyzed the ``average case'' complexity of the simplex method. Average case means that we are looking at a linear program
\[
\text{ Maximize } c\cdot {\bf x}, \text{ subject to } ,
\]
but the entries of $A$, $b$ and $c$ are considered random variables with respect to certain spherically symmetric probability distributions. In Borgwardt's model the simplex method runs in expected polynomial time in the size of the input. Smale shows that in his model, if one of the parameters $d$ and $n-d$ is fixed and the other is allowed to grow, the expected running time is only polylogarithmic. The latter was improved to constant by Megiddo, see~\cite{Megiddo:SurveyLPCompelxity} for details.

Even more surprising is the fact that \emph{every} linear program can be slightly perturbed to one that can be solved in polynomial time. Let us formalize this. Let $P$ be the feasibility polyhedron
\[
P = \{{\bf x} \in \R^d \mid \langle a_i, {\bf x}\rangle \leq b, (i=1,\ldots,n)\}
\]
of a certain linear program. 
If we replace the vectors $a_i \in \R^d$ and $b \in \R^n$ with independent Gaussian random vectors with means $\mu_i=a_i$ and $\mu=b$ (respectively), and standard deviations $\sigma \max_i \|(\mu_i,\mu)\|$ we say that we have \emph{perturbed $P$ randomly within a parameter $\sigma$}.
In~\cite{Spielman:WhySimplexUsually}, Spielman and Teng proved that the expected diameter of a linear program that is perturbed within a parameter $\sigma$ is polynomial in $d$, $n$, and $\sigma^{-1}$. 
In~\cite{Vershynin:BeyondHirsch}, Vershynin improved the bound to be polylogarithmic in  $n$ and polynomial only in $d$ and 
$\sigma^{-1}$.

\begin{theorem}[Vershynin~\cite{Vershynin:BeyondHirsch}]
\label{thm:vershynin}
If a linear program is perturbed randomly within a parameter $\sigma$, then its expected diameter of its feasibility polyhedron is $O(log^7 n (d^9 + d^3 \sigma^{-4})$.
\end{theorem}

As mentioned above, this result is not only structural. The simplex method can \emph{find} a path of that expected  length in the perturbed polyhedron.

\subsection{Some polytopes from combinatorial optimization}\label{sec:special}

There are some classes of polytopes of special interest and for the diameters of which we know polynomial upper bounds.

\subsubsection*{Small integer coordinates}

Of special importance in combinatorial optimization are the $0$-$1$ polytopes, in which every vertex has coordinates $0$ or $1$.They satisfy the Hirsch conjecture.

\begin{*theorem}[Naddef~\cite{Naddef:Hirsch01}] 
\label{thm:01hirsch}
If $P$ is a $0$-$1$ polytope then $$\diam(P) \leq \#\text{facets}(P) - \dim(P).$$
\end{*theorem}

As a generalization, Kleinschmidt and Onn~\cite{Kleinschmidt:Diameter} prove the following bound on the diameter of lattice polytopes in $[0,k]^d$. A polytope is called a \emph{lattice polytope} if every coordinate of every vertex is integral.

\begin{theorem}[Kleinschmidt and Onn~\cite{Kleinschmidt:Diameter}]
The diameter of a lattice polytope contained in $[0,k]^d$ cannot exceed $kd$.
\end{theorem}

However, existence of a polynomial pivot rule for the simplex method in $0$-$1$ polytopes is open. The proof of Theorem \ref{thm:01hirsch} constructs a short path from $u$ to $v$ only assuming that we know the coordinates of both.

\subsubsection*{Network-flow polytopes}

A network flow polytope is defined by an arbitrary directed graph $G=(V,E)$ with weights given to its vertices. Negative weights represent demands and positive weights represent supplies. A flow is an assignment of non-negative numbers to the edges so as to cancel all the demands and supplies. See~\cite{Cunningham:Theoretical-properties},~\cite{Goldfarb:Polynomial-simplex}, and~\cite{Orlin:PolytimeNetworkSimplex} for details.

For any network $G$ with $e$ edges and $v$ vertices, 
every sufficiently generic set of vertex weights produces a simple $(e-v+1)$-dimensional polytope with at most $2e$ facets. Its diameter has the following almost quadratic upper bound. The proof yields a polynomial time pivot rule for the simplex method on these polytopes.

\begin{theorem}[\cite{Cunningham:Theoretical-properties,Goldfarb:Polynomial-simplex,Orlin:PolytimeNetworkSimplex} ]
\label{thm:network-flow}
The diameter of the network flow polytope on a directed graph $G=(V,E)$ is $O(ev \log v)$. This, in turn,  is $O(n^2\log n)$, where $n$ is the number of facets of the polytope.
\end{theorem}

The matrices defining network flow polytopes are examples of \emph{totally unimodular} matrices, meaning that all its subdeterminants are $0$, $1$, or $-1$. Polytopes defined by these matrices still have polynomially bounded diameters, although the degree in the bound is much worse than the one for network flow polytopes:

\begin{theorem}[Dyer and Frieze~\cite{Dyer:RandomWalks}]
For any totally unimodular $n\times d$ matrix $A$
the diameter of the polyhedron  $\{{\bf x} \in \R^d : A{\bf x} \leq c \}$   is $O(d^{16}n^3(\log(dn))^3)$.
\end{theorem}

\subsubsection*{Transportation and dual transportation polytopes}

Given vectors $a \in \R^p$ and $b \in \R^q$,  the $p \times q$ transportation polytope defined by $a$ and $b$ is
the set of all ${p \times q}$ non-negative matrices with row sums given by $a$ and column sums given by $b$:
\[
T_{p,q}(a,b) = \{(x_{ij}) \in \R^{p \times q} \mid \sum_{j} x_{ij}=a_i, \sum_{i} x_{ij}=b_j, x_{ij} \geq 0\}. 
\]
As an example, the Birkhoff polytope, whose vertices are the permutation matrices, is the transportation polytope obtained with $p=q$ and $a=b=(1,\dots,1)$.

It is easy to show that (generically) $T_{p,q}(a,b)$ is a $(p-1)(q-1)$-dimensional polytope with at most $pq$ facets.
Thus, the Hirsch conjecture translates to its diameter being at most $p+q-1$.

Transportation polytopes are a special case of network flow polytopes; they arise when the network is a  complete bipartite graph on $p$ and $q$ nodes with all edges directed in the same direction. In particular, Theorem~\ref{thm:network-flow} gives an almost quadratic bound for their diameters. But  Brightwell et al.~\cite{Brightwell:LinearTransportation} have recently proved a \emph{linear} bound, with a multiplicative factor of eight. This has now been improved to:

\begin{theorem}[Hurkens~\cite{Hurkens:Diameter4p}]
The diameter of any $p \times q$ transportation polytope is at most $3(p+q-1)$.
\end{theorem}

In the context of linear programming, for every $d$-polyhedron with $n$ facets there is a dual $(n-d)$-polyhedron with the same number of facets. Every linear program can be solved in its ``primal'' or ``dual'' polyhedron. The optimum achieved is the same in both, but the complexity of the algorithm may not.

The linear programming duals of $p\times q$ transportation polytopes are $(p+q-1)$-polyhedra with $pq$ facets, Balinski~\cite{Balinski:DualTransportation} proved the Hirsch conjecture for them.

\begin{theorem}[Balinski~\cite{Balinski:DualTransportation}] 
Let $C$ be a $p \times q$ matrix. The diameter of the dual transportation polytope 
$D_{p,q}(C)$
is at most $(p-1)(q-1)$. This bound is the best possible and it yields a polynomial time dual simplex algorithm.
\end{theorem}

\subsubsection*{3-way transportation polytopes}

Another seemingly harmless generalization of transportation polytopes comes from considering $3$-way tables $(x_{ijk})_{ijk}\in\R^{p\times q\times r}$
instead of matrices. In fact, there are two different such generalizations. An \emph{axial $3$-way transportation polytope} consists of all non-negative tables with fixed sums in 2-d slices. A \emph{planar $3$-way transportation polytope} consists of all non-negative tables with fixed sums in 1-d slices. (The names seem wrong, but 
they reflect the fact that an axial transportation polytope is defined by three \emph{vectors} of sizes $p$, $q$ and $r$, and an axial one by three \emph{matrices} of sizes $p\times q$, $p\times r$, and $q\times r$). 

Despite their definition being so close to that of transportation polytopes, 3-way transportation polytopes are \emph{universal} in the following sense:

\begin{theorem}[De Loera and Onn~\cite{DO}]
Let $P$ be a rational convex polytope.
\begin{itemize}
\item
There is a $3$-way planar transportation polytope $Q$ isomorphic to $P$.
\item
There is a $3$-way axial transportation polytope $Q$ which has a face $F$ isomorphic to $P$.
\end{itemize}
In both cases there is a polynomial time algorithm to construct $Q$ (and $F$).
\end{theorem}

Isomorphic here means affinely (and rationally) equivalent. In particular, that the polytope $Q$ or its face $F$ have the same edge-graph as $P$. Thus, it was interesting to try to apply to the 3-way case the methods that gave polynomial upper bounds for the graphs of transportation polytopes. This was attempted in~\cite{DeLoera:GraphsTP}, where a quadratic upper bound was obtained but only for axial transportation polytopes. A generalization of this result to \emph{faces} of them or to \emph{planar} transportation polytopes would prove the polynomial Hirsch conjecture.

\begin{theorem}[De Loera, Kim, Onn, Santos~\cite{DeLoera:GraphsTP}]
The diameter of every $3$-way axial $p \times q \times r$ transportation polytope is at most $2(p+q+r-3)^2$.
\end{theorem}

\subsection{A continuous Hirsch conjecture}\label{sec:continuous}

Here we summarize some recent work of Deza, Terlaky and Zinchenko~\cite{Deza:Central-path,Deza:The-continuous-d-step,Deza:Curvature} in which they propose continuous analogues of the Hirsch and $d$-step conjectures related to the central path method---a variant of interior point methods---of linear programming. For a complete description of the method we refer the reader to~\cite{Boyd:ConvexOptimization, Renegar:A-Mathematical-View}.) The analogy comes from analyzing the total curvature $\lambda_c(P)$ of the central path with respect to a certain cost function $c$ for the polyhedron $P$. 
By analogy with $H(n,d)$, let $\Lambda(n,d)$ denote the largest total curvature of the central path over all polytopes $P$ of dimension $d$ defined by $n$ inequalities  and over all linear objective functions $c$.

It had been conjectured that $\lambda_c(P)$ is bounded by a constant for each dimension $d$, and that it grows at most linearly with varying $d$. Deza et al.~have disproved both statements: in \cite{Deza:Central-path}, they construct polytopes for which $\lambda_c(P)$ grows exponentially with $d$. More strongly, in \cite{Deza:Curvature} they construct a family of polytopes that show that $\lambda_c$ cannot be bounded only in terms of $d$:

\begin{theorem}[\cite{Deza:Curvature}]\label{thm:continuous-lower-bound}
For every fixed dimension $d\ge 2$, $\liminf_{n \rightarrow \infty} \frac{\Lambda(n,d)}{n} \geq \pi$.
\end{theorem}

Deza et al.~consider this result a continuous analogue of the existence of Hirsch-sharp polytopes. Motivated by this they pose the following conjecture:

\begin{conjecture}[Continuous Hirsch conjecture]
$\Lambda(n,d) \in O(n)$. That is, there is a constant $K$ such that $\Lambda(n,d)\le Kn$ for all $n$ and $d$.
\end{conjecture}

Theorem~\ref{thm:continuous-lower-bound} says that if the continuous Hirsch conjecture is true, then it is  tight, modulo a constant factor.
Deza et al.~also conjecture a continuous variant of the $d$-step conjecture, and show it to be equivalent to the continuous Hirsch conjecture, thus providing an analogue of Theorem~\ref{thm:dstep-hirsch}:

\begin{conjecture}[Continuous $d$-step conjecture]
The function $\Lambda(2d,d)$ grows linearly in its input. That is to say, $\Lambda(2d,d)$ is $O(d)$.
\end{conjecture}

\begin{theorem}[\cite{Deza:The-continuous-d-step}]
The continuous Hirsch conjecture is equivalent to the continuous $d$-step conjecture. That is, if $\Lambda(2d,d) \in O(d)$ for all $d$, then $\Lambda(n,d)\in O(n)$ for all $d$ and $n$.
\end{theorem}

The best upper bound known for $\Lambda(n,d)$ is $O(n^d)$, derived from Theorem~\ref{thm:arrangement-curvature} below. This theorem refers to the central path curvature for hyperplane arrangements, as studied by Dedieu, Malajovich and Shub~\cite{Dedieu}. 

 An arrangement $\mathcal{A}$ of $n$ hyperplanes in dimension $d$ is called \emph{simple} if every $n$ hyperplanes intersect at a unique point. It is easy to show that any simple arrangement of $n$ hyperplanes in $\real^d$ has exactly $s=\binom{n-1}{d}$ bounded full-dimensional cells.
For a simple arrangement $\mathcal{A}$ with bounded cells $P_1,\dots,P_s$
and a given objective function $c$, Dedieu et al. consider the quantity:
$\lambda_c(\mathcal{A})= \frac{1}{s} \sum_{i=1}^{s} \lambda_c(P_i)$.
That is, the average total curvature of central paths of all bounded cells in the arrangement. They prove:
\begin{theorem}[\cite{Dedieu}]
\label{thm:arrangement-curvature}
$\lambda_c(\mathcal{A}) \leq 2\pi d$, for every simple arrangement.
\end{theorem}
Put differently, even if individual cells can give total curvature linear in $n$  by Theorem~\ref{thm:continuous-lower-bound}, the average over all cells of a given arrangement is bounded by a function of $d$ alone.

Turning the analogy back to polytope graphs,  Deza et al.~\cite{Deza:Curvature} consider the average diameter of the graphs of all bounded cells in a simple arrangement $\mathcal{A}$. Denote it $\diam(\mathcal{A})$ and let $\mathcal{H}(n,d)$ be the maximum of $\diam(\mathcal{A})$ over all simple arrangements defined by $n$ hyperplanes in dimension $d$. They relate $\mathcal{H}(n,d)$ to the Hirsch conjecture, as follows:
%
%
\begin{proposition}[\cite{Deza:Curvature}]
The Hirsch conjecture implies $\mathcal{H}(n,d) \leq d + \frac{2d}{n-1}$.
\end{proposition}

\section{Constructions}
\label{sec:constructions}
We now move to interesting constructions of polytopes motivated or related to the Hirsch conjecture. All the proofs that are not included in this section, plus additional comments, can be found in~\cite{Kim-Santos-companion}.

\subsection{The wedge operation}\label{sec:wedging}
Wedging is a very basic, yet extremely fruitful, operation that one can 
do to a polytope. Its simplicial 
counter-part is the \emph{one-point suspension}, see~\cite{Kim-Santos-companion}.

Roughly speaking, the wedge of $P$ at a facet $F$ of it is the polytope, of one dimension more, obtained gluing two copies of $P$ along $F$. See Figure~\ref{fig:wedge} for an example. More formally,  let $f(x)\le b$ be an inequality defining the facet $F$. The wedge of $P$ over $F$ is the polytope 
\[
\wed_F(P) := P \times [0,\infty) \cap \{ (x,t)\in \R^d\times \R : f(x) + t \le b \}.
\]
Put differently, $\wed_F(P)$ is formed by intersecting the half-cylinder $C:=P \times [0,\infty)$ with a closed half-space $J$ in $\R^{d+1}$ such that:
\begin{itemize}
\item the intersection $J \cap C$ is bounded and has nonempty interior, and
\item $\partial J \cap C = F$.
\end{itemize}

\begin{lemma}
\label{lemma:wedge}
Let $P$ be a $d$-polytope with $n$ facets. Let  $\wed_F(P)$ be its wedge on a certain facet $F$. Then, $\wed_F(P)$ has dimension $d+1$, $n+1$ facets, and
\[
\diam(\wed_F(P)) \ge \diam(P).
\]
\end{lemma}

\begin{proof}
The wedge  increases both the dimension and the number of facets by one. Indeed, $\wed_F(P) $ has a vertical facet projecting to each facet of $P$ other than $F$, plus the two facets that cut the cylinder $P\times \R$, and whose intersection projects to $F$. 
%
\begin{figure}[hbt]
\begin{center}
\includegraphics[scale=0.60]{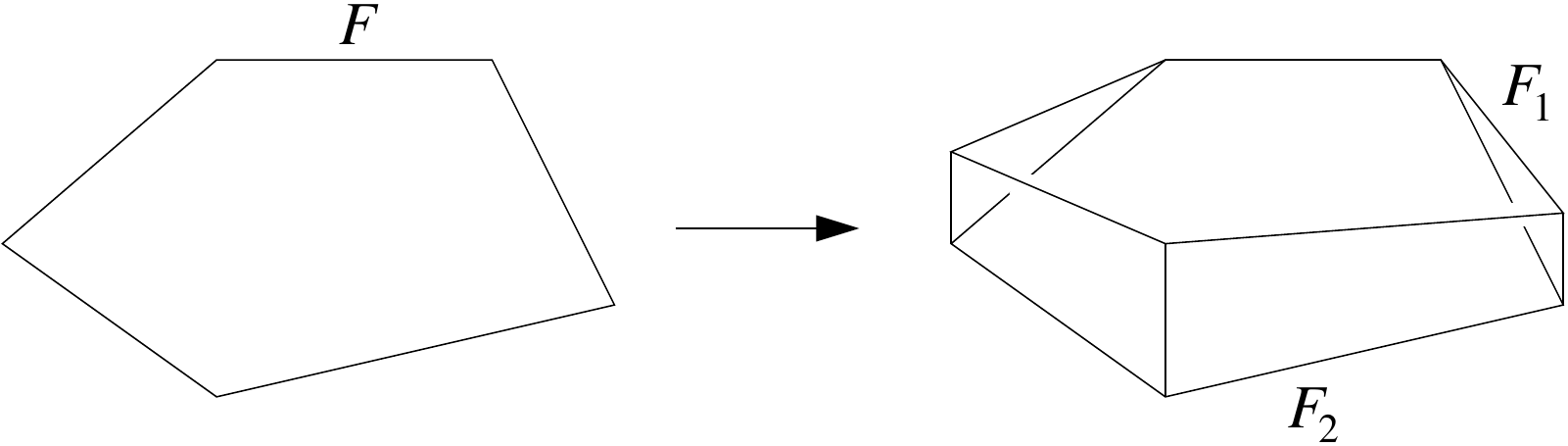}
\caption{A $5$-gon and a wedge on its top facet}\label{fig:wedge}
\end{center}
\end{figure}
The diameter of $\wed_F(P)$ is at least that of $P$, since every edge of $\wed_F(P)$ projects either to an edge of $P$ or to a vertex of $P$.
\end{proof}

In particular, if $P$ is Hirsch-sharp then $\wed_F(P)$ is either Hirsch-sharp or a counterexample to the Hirsch conjecture. The properties that $P$ would need for the latter to be the case will be made explicit in Remark~\ref{rem:non-hirsch-wedge}.

As a corollary of Lemma~\ref{lemma:wedge} we get that in order to prove (or disprove) the Hirsch conjecture it is sufficient to restrict attention to the case when the number of facets equals twice the dimension:

\begin{theorem}[Klee-Walkup~\cite{Klee:d-step}]
\label{thm:dstep-hirsch}
$H(k+d,d) \le H(2k,k)$, with equality if (but not necessarily only if) $k<d$.
%
\end{theorem}

\begin{proof}
By Lemma~\ref{lemma:wedge},
\begin{equation}\label{inequalities}
H(n,d)\le H(n+1,d+1), \qquad \forall n,d
\end{equation}
so we only need to show that
\begin{equation}\label{inequalities2}
H(n,d)\le H(n-1,d-1)  \qquad \forall n<2d.
\end{equation}

Let $P$ be a polytope with $n < 2d$ and let $u$ and $v$ be vertices of it.
Since each vertex is incident to at least $d$ facets,  $u$ and $v$ lie in a common facet. This facet $F$ has dimension $d-1$, and each facet of it is the intersection of $F$ with another facet of $P$. Hence, $F$ has at most $n-1$ facets itself.
Since every path on $F$ is also a  path on $P$, we get \eqref{inequalities2}.
\end{proof}

\subsection{The $d$-step and non-revisiting conjectures}\label{sec:equivalences}



The intuition behind the Hirsch conjecture is that  to go from vertex $u$ to vertex $v$ of a polytope $P$, one does not expect to have to enter and leave the same facet several times.
This suggests the following conjecture:

\begin{conjecture}[The non-revisiting conjecture]\label{conj:nonrevisiting}
Let $P$ be a simple polytope. Let $u$ and $v$ be two arbitrary vertices of $P$. Then, there is a path from $u$  to $v$ which at every step enters a facet of $P$ that was not visited before.
\end{conjecture}

Paths with the conjectured property, that they do not revisit any facet, are called \emph{non-revisiting paths}. (In the literature, they are also called $W_v$ paths and Conjecture~\ref{conj:nonrevisiting} is also known as the $W_v$ conjecture.) Non-revisiting paths are never longer than $n-d$: at each step, we must enter a different facet, and the $d$ facets that the initial vertex lies in cannot be among them. Thus, the non-revisiting conjecture implies the Hirsch conjecture. It turns out both are equivalent. A first step in the proof is the following analogue of 
Theorem~\ref{thm:dstep-hirsch} for the non-revisiting conjecture:

\begin{theorem}
\label{thm:nonrevisiting}
If all $k$-polytopes with $2k$ facets has the non-revisiting property, then the same holds for all $d$-polytopes with $d+k$ facets, for all $d$.
\end{theorem}

\begin{proof}
Let $P$ be a polytope with $n\ne 2d$ and suppose it does not have the non-revisiting property. That is, there are vertices $u$ and $v$ such that every path from $u$ to $v$ revisits some facet that it previously abandons.
We will construct another polytope $P'$ without the non-revisiting property and with:
\begin{itemize}
\item One less facet and dimension than $P$ if $n < 2d$, and
\item One more facet and dimension than $P$ if $n > 2d$.
\end{itemize}

\begin{figure}[hbt]
  \begin{center}
    \includegraphics[scale=0.83]{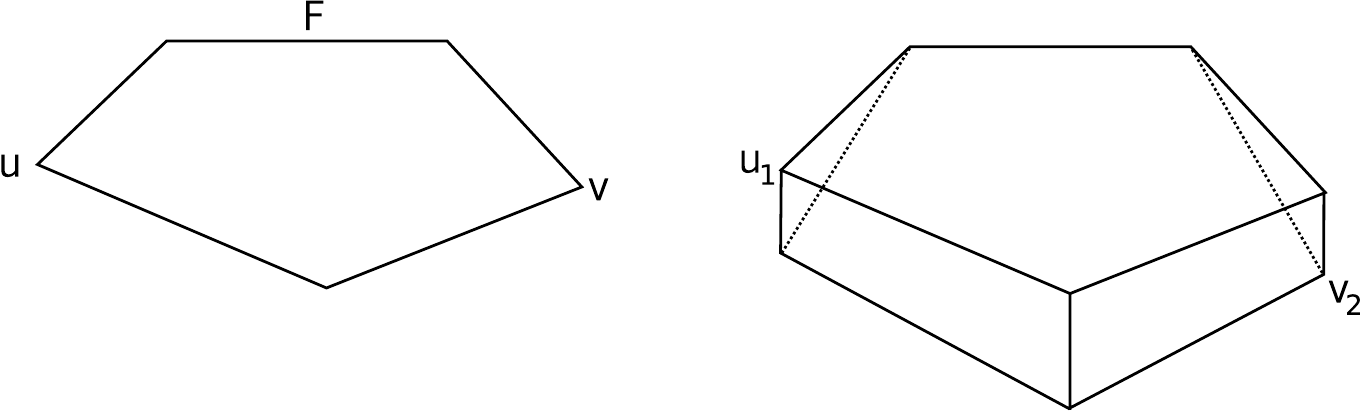}
    \caption{The pentagon $P$ and the wedge $P'=\wed_F(P)$ over its facet $F$: the upper pentagonal facet of $P'$ is $F_1$ and the lower pentagonal facet is $F_2$.}
       \label{fig:wedge-when-n-larger-2d}
  \end{center}
\end{figure}

In the first case, $u$ and $v$ lie in a common facet $F$ and we simply let $P'=F$. 
In the second case, let $F$ be a facet not containing $u$ nor $v$ and let $P'=\wed_F(P)$ be the wedge over $F$. Let $F_1$ and $F_2$ be the two facets of $P'$ whose intersection projects to $F$. Let $u_1$ and $v_2$ be the vertices of $P'$ that project to $u$ and $v$ and lie, respectively, on $F_1$ and $F_2$ (see Figure~\ref{fig:wedge-when-n-larger-2d}). Now, consider a path from $u_1$ to $v_2$ on $P'$ and project it to a path from $u$ to $v$ on $P$:
\begin{itemize}
\item If the path on $P$ revisits a facet (call it $G$) other than $F$, then the path on $P'$ revisits the facet that projects to $G$.

\item If the path on $P$ revisits $F$, then the path on $P'$ revisits either $F_1$ or $F_2$.
\end{itemize}
\end{proof}

\begin{remark}
\label{rem:non-hirsch-wedge}
\rm
In the proof of Lemma~\ref{lemma:wedge} we noted that, applied to a Hirsch-sharp polytope $P$ the wedge operator produced either another Hirsch-sharp polytope or a counterexample to the Hirsch conjecture. The last proof shows that the latter can happen only if $P$ does not have the non-revisiting property.
\end{remark}

Theorems~\ref{thm:dstep-hirsch} and~\ref{thm:nonrevisiting} say that both in the Hirsch and the non-revisiting conjectures the crucial case is that of $n=2d$. It is not surprising then that they are equivalent, since in this case they both almost restrict to the following:

\begin{conjecture}[The $d$-step conjecture]\label{conj:d-step}
Let $P$ be a simple $d$ polytope with $2d$ facets and let $u$ and $v$ be two complementary vertices (i.e., vertices not lying in a common facet). Then, there is a path of length $d$ from $u$ to $v$.
\end{conjecture}

There is still something to be proved, though. If $u$ and $v$ are not complementary vertices in a $d$-polytope with $2d$ facets then the $d$-step conjecture does not directly imply the other two. But in this case $u$ ad $v$ lie in a common facet, so the proof of equivalence is not hard to finish via induction:

\begin{*theorem}[Klee-Walkup~\cite{Klee:d-step}]
\label{thm:dstep-nonrevisiting}
The Hirsch, non-revisiting, and $d$-step Conjectures~\ref{conj:hirsch},~\ref{conj:nonrevisiting}, and~\ref{conj:d-step} are equivalent.
\end{*theorem}

\subsection{The Klee-Walkup polytope $Q_4$} \label{sec:Q4}

In their seminal 1967 paper~\cite{Klee:d-step} on the Hirsch conjecture and related issues, Klee and Walkup describe a $4$-polytope $Q_4$ with nine facets and diameter five. Innocent as this might look, this first ``non-trivial'' Hirsch-sharp polytope is at the basis of the construction of every remaining Hirsch-sharp polytope known to date (see Section~\ref{sec:many-hirsch-sharp}). It is also instrumental in disproving the unbounded and monotone variants of the Hirsch conjecture, which we will discuss in Section~\ref{sec:unbounded-monotone}. Moreover, 
its existence is something of an accident: Altshuler, Bokowski and Steinberg~\cite{ABS-3spheres} list all  combinatorial types of simplicial spheres with nine vertices (there are $1296$, $1142$ of them polytopal); among them, the polar of $Q_4$ is the only one that is Hirsch-sharp.

Here we describe $Q_4$ in the polar view. That is, we will describe a simplicial $4$-polytope $Q_4^*$ with nine vertices and show that its ridge-diameter is five. The vertices of $Q_4^*$ are:

\[
\begin{tabular}{lll}
$a:= (-3,3,1,2)$,&&$e:= (3,3,-1,2)$,\\
$b:=(3,-3,1,2),$&&$f:=(-3,-3,-1,2),$\\
$c:=(2,-1,1,3)$,&&$g:=(-1,-2,-1,3)$,\\
$d:=(-2,1,1,3)$,&&$h:=(1,2,-1,3)$,\\
&$w:=(0,0,0,-2)$.
\end{tabular}
\]
[The simple polytope $Q_4$ is obtained  converting each vertex $v$ of $Q_4^*$ into an
inequality $v \cdot {\bf x}\le 1$. For example, the inequality corresponding to vertex $a$ above is $-3x_1+3 x_2+ x_3+ 2x_4\le 1$].


The key property of this polytope is that:
\begin{theorem}[Klee-Walkup~\cite{Klee:d-step}]
\label{thm:klee-walkup}
Any path in $Q_4^*$ from the tetrahedron $abcd$ to the tetrahedron $efgh$ needs at least five steps.
\end{theorem}

To prove this, you may simply  input these coordinates into any software able to compute the (dual) graph of a polytope. Our suggestion for this is {\tt polymake}~\cite{polymake}.
%
But we believe that fully understanding this polytope can be the key to the construction of counter-examples to the Hirsch conjecture, so it is worth presenting a hybrid computer-human proof. It is worth mentioning that the coordinates we use for $Q_4^*$, much smaller than the original ones in~\cite{Klee:d-step}, were obtained as a by-product of the description of $Q_4^*$ contained in this proof.

\begin{proof}
Paths through some intermediate tetrahedron containing the vertex $w$ necessarily have at least five steps: apart of the step that introduces $w$, four more are needed to introduce, one by one, the four vertices  $e$, $f$, $g$ and $h$. 

This means we can concentrate on the subcomplex $K$ of $\partial Q_4^*$ consisting of tetrahedra that do not use $w$. This subcomplex is called the \emph{anti-star} of $w$ in $\partial Q_4^*$. We claim (without proof, here is where you need your computer) that this subcomplex consists of the $15$ tetrahedra in Figure~\ref{fig:dualK}.
\begin{figure}[hbt]
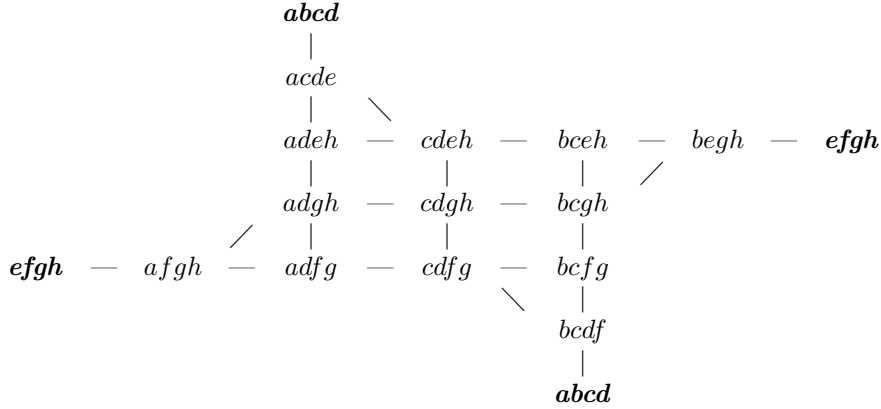

\[
\begin{matrix}
 & & \textbf{\textit{abcd}} && && &&&\\
 && | & & && &&  &\\
 & & acde && && &&&\\
  && | & \diagdown & &&&& &\\
  && adeh &\text{---} & cdeh &\text{---} & bceh &\text{---} & begh&\text{---}\quad  \textbf{\textit{efgh}} \\
  && | && | && | & \diagup & &\\
  && adgh &\text{---} & cdgh &\text{---} & bcgh  && &\\
   & \diagup & | && | && | &&&\\
 \textbf{\textit{efgh}}\quad\text{---}\quad afgh & \text{---} & adfg &\text{---} & cdfg &\text{---} & bcfg &&&\\
   && && & \diagdown & |  &&&\\
  && && & & bcdf & &&\\
   && && & & |  &&&\\
     && && & & \textbf{\textit{abcd}} &  &&\\
 \end{matrix}
\]
\caption{The dual graph of the subcomplex $K$}\label{fig:dualK}
\end{figure}

Figure~\ref{fig:dualK} shows adjacencies among tetrahedra; that is, it shows the dual graph of $K$. The two tetrahedra $abcd$ and $efgh$ that we want to join are in boldface and appear repeated in the figure, to better reflect symmetry. The proof finishes by noticing that there is no intermediate tetrahedron that can be reached in two steps from both $abcd$ and $efgh$. Hence, five steps are needed to go from one to the other.
\end{proof}

\subsection{Many Hirsch-sharp polytopes?}\label{sec:hirsch-sharp}

Recall that we call a $d$-polytope (or polyhedron) with $n$ facets \emph{Hirsch-sharp} if its diameter is exactly $n-d$, as happens with the Klee-Walkup polytope of the previous section. Here we describe several ways to construct them.

\subsubsection*{Trivial Hirsch-sharp polytopes}\label{sec:examples}

Constructing Hirsch sharp $d$-polytopes with any number of facets not exceeding $2d$ is easy. For this reason we call such Hirsch-sharp polytopes \emph{trivial}:

\begin{enumerate}


\item 
{\bf Product.}
If $P$ and $Q$ are Hirsch-sharp, then so is their Cartesian product $P\times Q$. Indeed, the dimension, number of facets, and diameters of $P\times Q$ are the sum of those of $P$ and $Q$. For the diameter, if we want to go from vertex $(u_1,v_1)$ to vertex $(u_2, v_2)$ we can do so by going from $(u_1,v_1)$ to $(u_2,v_1)$ along $P\times \{v_1\}$ and then to $(u_2,v_2)$ along $\{u_2\}\times Q$; there is no better way.

In particular,  any product of \emph{simplices} of any dimension is Hirsch-sharp. 
The dimension of  $\Delta_{i_1} \times \cdots \times \Delta_{i_k}$, where $\Delta_i$ denotes the $i$-simplex,
 is $\sum_{j=1}^k i_j$,  its number of facets is  $\sum_{j=1}^k (i_j+1)$, and its diameter is $k$. 

\begin{corollary}
\label{coro:nlessthan2d}
For every $d< n \le 2d$ there are simple $d$-polytopes with $n$ facets and diameter $n-d$.
\end{corollary}

\begin{proof}
Let $k=n-d\le d$ and let $i_1,\dots,i_k$ be any partition of $d$ into $k$ positive integers (that is, $i_1 + \cdots +i_k=d$). Let  $P=\Delta_{i_1} \times \cdots \times \Delta_{i_k}$.
\end{proof}

\item {\bf Intersection of two affine orthants.}
Let $k = n-d \le d$ and let $u$ be the origin in $\R^d$. Let $v=(1,\dots,1,0,\dots,0)$ be the point whose first $k$ coordinates are $1$ and whose remaining $d-k$ coordinates are $0$.

Consider the polytope $P$ defined by the following $d+k$ inequalities:
\[
x_i \ge 0,\quad \forall i; \qquad \psi_j({\bf x})\ge 0, \quad j=1,\dots, k,
\]
where the $\psi_i$ are affine linear functionals that vanish at $v$ and are positive at $u$. No matter what choice we make for the $\psi_j$'s, as long as they are sufficiently generic to make $P$ simple, $P$ will have diameter (at least) $k$; to go from $v$ to $u$ we need to enter the $k$ facets $x_j=0$, $j=1,\dots, k$, and each step gets you into at most one of them.  In principle, $P$ may be an unbounded polyhedron; but if one of the $\psi_j$'s is, say, $k- \sum x_i $, then it will be bounded.

\end{enumerate}


Hirsch-sharp unbounded polyhedra with any number of facets are also easy to obtain:

\begin{*proposition}
For every $n \ge d$ there are simple $d$-polyhedra with $n$ facets and diameter $n-d$.
\end{*proposition}



\subsubsection*{Non-trivial Hirsch-sharp polytopes}\label{sec:many-hirsch-sharp}


\begin{*theorem}[Fritzsche-Holt-Klee~\cite{Fritzsche99morepolytopes,Holt:Hsharpd7,Holt:Many-polytopes}]
\label{thm:hirsch-sharp}
Hirsch-sharp $d$-polytopes with $n$ facets exist in at least the following cases:
(1) $n\le 3d-3$; and
(2) $d\ge 7$.
\end{*theorem}

Both parts are proved using the Klee-Walkup polytope $Q_4$ as a starting block, from which more complicated polytopes are obtained. In a sense, \emph{$Q_4$ is the only non-trivial Hirsch-sharp polytope we know of}.

The case $n\le 3d-3$ was first proved in 1998~\cite{Holt:Many-polytopes}, and follows from the iterated application of the next lemma to the Klee-Walkup polytope $Q_4$. Part 2 was proved in~\cite{Fritzsche99morepolytopes} for $d\ge 8$ and was improved to $d\ge 7$ in~\cite{Holt:Hsharpd7}. We sketch its proof in~\cite{Kim-Santos-companion}.

\begin{lemma}[Holt-Klee~\cite{Holt:Many-polytopes}]
\label{lem:3d-3}
If there are Hirsch-sharp $d$-polytopes with $n > 2d$ facets, then there are also Hirsch-sharp $(d+1)$-polytopes with $n+1$, $n+2$, and $n+3$ facets.
\end{lemma}

\addtocounter{theorem}{1}

\begin{proof}
Let $u$ and $v$ be vertices at distance $n-d$ in a simple $d$-polytope with $n$-facets. Let $F$ be a facet not containing any of them, which exists since $n>2d$. When we wedge on $F$ we get two edges $u_1u_2$ and $v_1v_2$ with the properties that the distance from any $u_i$ to any $v_i$ is again (at least) $d$. We can then truncate one or both of $u_1$ and $v_1$ to obtain one or two more facets in a polytope that is still Hirsch-sharp. See Figure~\ref{fig:sharp-3d-3}.
\end{proof}

\begin{figure}[hbt]
\begin{center}
\includegraphics[scale=0.60]{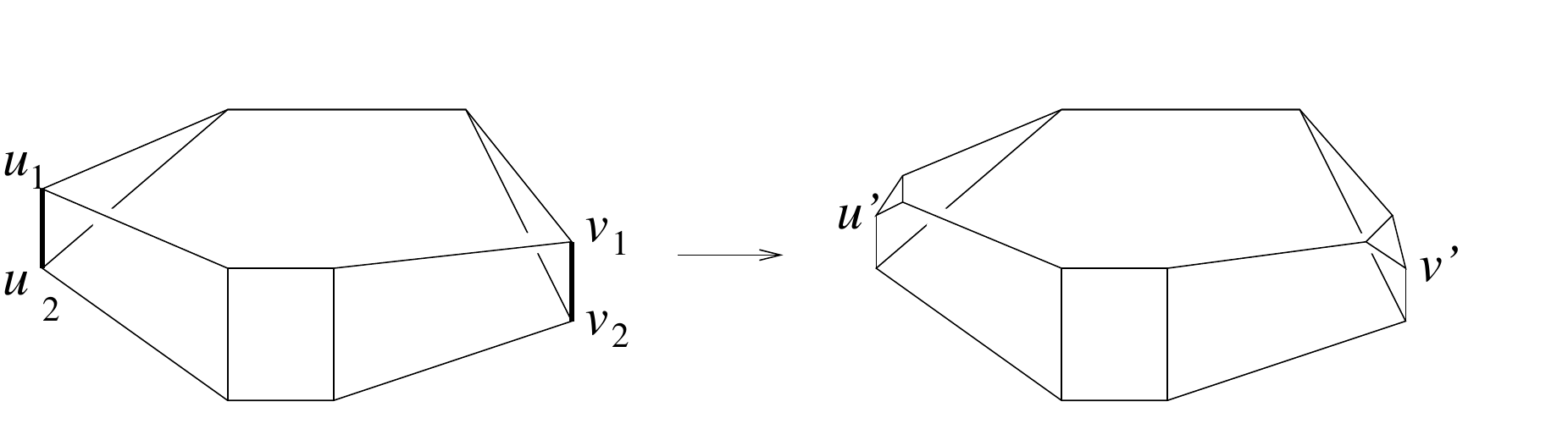}
\caption{After wedging in a Hirsch-sharp polytope we can truncate twice}\label{fig:sharp-3d-3}
\end{center}
\end{figure}

Hirsch-sharp polytopes of dimensions two and three exist only when $n\le 2d$ (see Section~\ref{sec:small-dimension}). Existence of Hirsch-sharp polytopes with many facets in dimensions four to six remains open. We do know that they do not exist in dimension four with 10, 11 or 12 facets (see Theorem~\ref{thm:hirschforsmalln}), which may well indicate that $Q_4$ is the only
Hirsch-sharp $4$-polytope.



\subsection{The unbounded and monotone Hirsch conjectures}\label{sec:unbounded-monotone}

In the Hirsch conjecture as we have stated it, we only consider bounded polytopes. However, in the context of linear programming the feasible region may well not be bounded, so the conjecture is equally relevant for \emph{unbounded} polyhedra. In fact, that is how W.~Hirsch originally posed the question. 

Moreover, for the simplex method in linear programming one follows \emph{monotone paths}: starting at an initial vertex $u$ one does pivot steps (that is, one crosses edges) always increasing the value of the linear functional $\phi$ to be maximized, until one arrives at a vertex $v$ where no pivot step gives a greater value to $\phi$. Convexity then implies that $v$ is the global maximum for $\phi$ in the feasible region.
This raises the question whether a \emph{monotone} variant of the Hirsch conjecture holds: given two vertices $u$ and $v$ of a polyhedron $P$ and a linear functional that attains its maximum on $P$ at $v$, is there a $\phi$-monotone path of edges from $u$ to $v$ whose length is at most $n-d$? 

Both the unbounded and monotone variants of the Hirsch conjecture fail, and both proofs use the Klee-Walkup Hirsch-sharp polytope $Q_4$  described in Section~\ref{sec:Q4}. In fact, knowing the mere existence of such a polytope is enough. The proofs do not use any property of $Q_4$ other than the fact that it is Hirsch-sharp, simple, and has $n>2d$. Simplicity is not a real restriction since it can always be obtained without decreasing the diameter (Lemma~\ref{lemma:simple}). The 
inequality $n>2d$, however, is essential.

\begin{theorem}[Klee-Walkup~\cite{Klee:d-step}]
\label{thm:unbounded}
There is a simple unbounded polyhedron $\tilde Q_4$ with eight facets and dimension four and whose graph has diameter 5.
\end{theorem}

\begin{proof}
Let $Q_4$ be the simple Klee-Walkup polytope with nine facets,  and let $u$ and $v$ be vertices of  $Q_4$ at distance five from one another.
By simplicity, $u$ and $v$ lie in (at most) eight facets in total and there is (at least) one facet $F$ not containing $u$ nor $v$. 
Let $\tilde Q_4$ be the unbounded polyhedron obtained by a projective transformation that sends this ninth facet to infinity. The graph of $\tilde Q_4$ contains both $u$ and $v$, and is a subgraph of that of $\tilde Q_4$, hence its diameter is still at least five. See Figure~\ref{fig:unbounded} for a schematic rendition of this idea.
\end{proof}

\begin{figure}[htb]
\begin{center}
\includegraphics[width=4in]{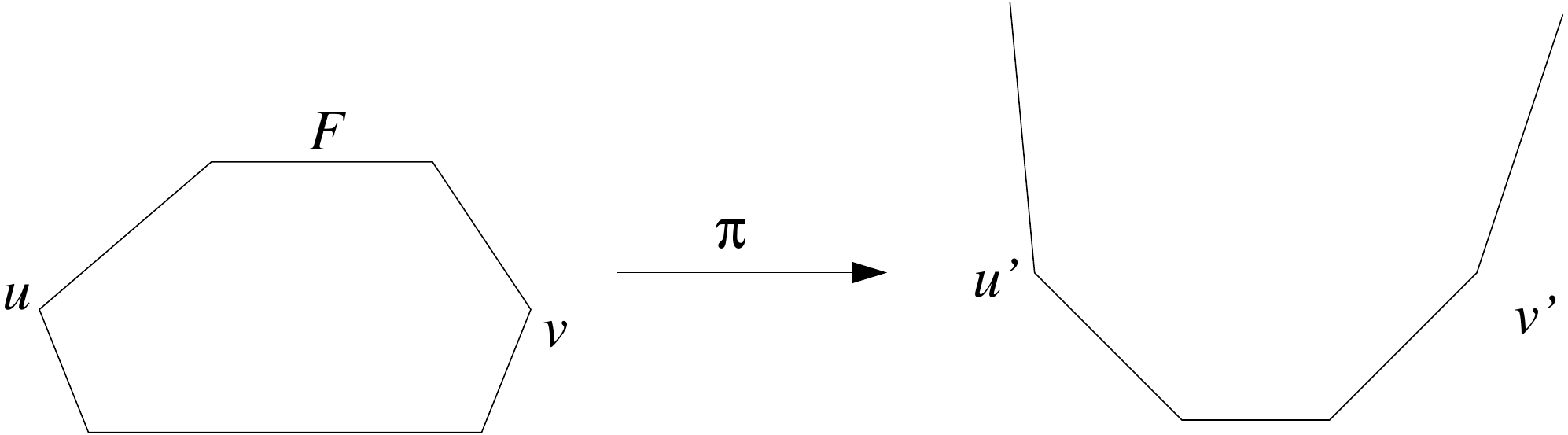}
\caption{Disproving the unbounded Hirsch conjecture}
\label{fig:unbounded}
\end{center}

\end{figure}

\begin{remark}
\rm
The ``converse'' of the above proof also works: from any non-Hirsch unbounded $4$-polyhedron $\tilde Q$ with eight facets, one can build a bounded $4$-polytope with nine facets and diameter five, as follows:

Let $u$ and $v$ be vertices of $\tilde Q$ at distance five from one another.
Construct the polytope $Q$ by cutting $\tilde Q$ with a hyperplane that leaves all the vertices of $\tilde Q$ on the same side. This adds a new facet and changes the graph, by adding new vertices and edges on that facet. But $u$ and $v$ will still be at distance five: to go from $u$ to $v$ either we do not use the new facet $F$ that we created (that is, we stay in the graph of $\tilde Q_4$) or we use a pivot to enter the facet $F$ and at least another four to enter the four facets containing $v$.
\end{remark}

We now turn to the monotone variant of the Hirsch conjecture:

\begin{*theorem}[Todd~\cite{Todd:MonotonicBoundedHirsch}]
\label{thm:monotone}
There is a simple $4$-polytope $P$ with eight facets, two vertices $u$ and $v$ of it, and a linear functional $\phi$ such that:
\begin{enumerate}
\item $v$ is the only maximal vertex for $\phi$.
\item Any edge-path from $u$ to $v$ and monotone with respect to $\phi$ has length at least five.
\end{enumerate}
\end{*theorem}
In both the constructions of Theorems~\ref{thm:unbounded} and~\ref{thm:monotone} one can glue several copies of the initial block $Q_4$ to one another,
increasing the number of facets by four and the diameter by five, per $Q_4$ glued:

\begin{theorem}[Klee-Walkup, Todd]
\begin{enumerate}
\item There are unbounded $4$-polyhedra with $4+4k$ facets and diameter $5k$, for every $k\ge 1$.

\item There are bounded $4$-polyhedra with $4+4k$ facets and vertices $u$ and $v$ of them with the property that any monotone path from $u$ to $v$ with respect to a certain linear functional $\phi$ maximized at $v$ has length at least $5k$.
\end{enumerate}
\end{theorem}

This leaves the following open questions:

\begin{itemize}
\item Improve these constructions so as to get the ratio ``diameter versus facets'' bigger than $5/4$. A ratio bigger than two for the unbounded case would probably yield counter-examples to the bounded Hirsch conjecture.

\item Ziegler~\cite[p.~87]{Ziegler:LecturesPolytopes} poses the following \emph{strict} monotone Hirsch conjecture: ``for every linear functional $\phi$ on a $d$-polytope with $n$ facets there is a $\phi$-monotone path of length at most $n-d$ from the minimal to the maximal vertex''. Put differently, in the monotone Hirsch conjecture we add the requirement that not only $v$ but also $u$ has a supporting hyperplane where $\phi$ is constant.
\end{itemize}

\subsection{The topological Hirsch conjecture is false}\label{sec:topological}

Another natural variant of the Hirsch conjecture is topological. Since (the boundary of) every simplicial $d$-polytope is a topological triangulation of the $(d-1)$-dimensional sphere, we can ask whether the simplicial version of the Hirsch conjecture, the one where we walk from simplex to simplex rather than from vertex to vertex, holds for arbitrary triangulations of spheres. The first counterexample to this statement was found by Walkup in 1979 (see~\cite{Walkup}), and a simpler one was soon constructed by Walkup and Mani~\cite{Mani:A-3-sphere-counterexample}.

Both constructions are based on the equivalence of the Hirsch conjecture to the non-revisiting conjecture (Theorem~\ref{thm:dstep-nonrevisiting}). The proof of the equivalence is purely combinatorial, so it holds true for topological spheres. Walkup's initial example is a $4$-sphere without the non-revisiting property, and Mani and Walkup's is a $3$-sphere:

\begin{*theorem}[Mani-Walkup~\cite{Mani:A-3-sphere-counterexample}]
\label{thm:mani-walkup}
There is a triangulated 3-sphere with 16 vertices and without the non-revisiting property. 
Wedging on it eight times yields an $11$-sphere with $24$ vertices and with ridge-graph diameter at least $13$.
\end{*theorem}

This triangulated $3$-sphere  would give a counterexample to the Hirsch conjecture if it were \emph{polytopal}. That is, if it
were combinatorially isomorphic to the boundary complex of a four-dimensional simplicial polytope. Altshuler~\cite{Altshuler:NotPolytopal} has shown that (for the explicit completion of the subcomplex $K$ given in~\cite{Mani:A-3-sphere-counterexample}) this is not the case. As far as we know, it remains an open question to show that \emph{no completion of $K\cup \{abcd,mnop\}$} to the 3-sphere is polytopal, but we believe that to be the case. Even more strongly,
 we believe that $K\cup \{abcd,mnop\}$ cannot be embedded in $\real^3$ with linear tetrahedra, a necessary condition for polytopality  by the well-known Schlegel construction~\cite{Ziegler:LecturesPolytopes}.

As in the monotone and bounded cases, several copies of the construction can be glued to one another. Doing so provides triangulations of the 11-sphere with $12 +12k$ vertices and diameter at least $13k$, for any $k$.

\subsection*{Acknowledgments}
This paper grew out of several conversations during the second author's sabbatical leave at UC Davis in 2008 and the authors' participation in the I-Math DocCourse on Discrete and Computational Geometry at the \emph{Centre de Recerca Matem\`atica} in 2009. We thank both institutions for hosting us and the financial support from the National Science Foundation and  the Spanish Ministry of Science.  

We also thank David Bremner, Jes\'us De Loera, Antoine Deza, J\"org Rambau, G\"unter M. Ziegler  and the anonymous referee for their valuable comments on the first version of the paper.

\vskip.5cm
\noindent {\small Edward D. Kim}\newline
\emph{Department of Mathematics}\newline
\emph{University of California, Davis. Davis, CA 95616, USA}\newline
\emph{email: }\url{ekim@math.ucdavis.edu}\newline
\emph{web: }\url{http://www.math.ucdavis.edu/~ekim/}

\vskip.5cm
\noindent {\small Francisco Santos}\newline
\emph{Departamento de Matem\'aticas, Estad\'istica y Computaci\'on}\newline
\emph{Universidad de Cantabria, E-39005 Santander, Spain}\newline
\emph{email: }\url{francisco.santos@unican.es}\newline
\emph{web: }\url{http://personales.unican.es/santosf/}



\end{document}